\documentclass{article}
\usepackage[T1]{fontenc}
\usepackage{a4wide}
\usepackage{graphicx} 
\usepackage{systeme}
\usepackage{amsmath}
\usepackage{amssymb}
\usepackage{cite}
\usepackage[hidelinks]{hyperref}
\usepackage{xcolor}
\usepackage{booktabs} 
\usepackage[caption=false,font=footnotesize]{subfig}
\usepackage{flushend}
\usepackage{amsthm}
\usepackage[mathscr]{euscript}
\newtheorem{theorem}{Theorem}
\newtheorem{assum}{Assumption}
\newtheorem{lemma}[theorem]{Lemma}
\newtheorem{prop}[theorem]{Proposition}

\newtheorem{remark}[theorem]{Remark}

\DeclareMathAlphabet{\mathbbb}{U}{bbold}{m}{n}

\renewcommand{\leq}{\leqslant}
\renewcommand{\geq}{\geqslant}
\renewcommand{\le}{\leqslant}
\renewcommand{\ge}{\geqslant}

\newcommand{\dt}{\mathrm{d}t}

\newcommand{\ab}[1]{\begingroup\color{black}#1\endgroup}

\newcommand{\modif}[1]{\begingroup\color{black}#1\endgroup}

\newcommand{\R}{\mathbb{R}}
\newcommand{\N}{\mathbb{N}}
\newcommand{\C}{\mathbb{C}}

\newcommand{\D}{\mathbb{D}}
\newcommand{\K}{\mathbb{K}}
\newcommand{\T}{\mathbb{T}}

\newcommand{\fhat}{\widehat{f}}
\newcommand{\ghat}{\widehat{g}}
\newcommand{\Nhat}{\widehat{N}}

\newcommand{\Ntarget}{\Lap N_T}
\newcommand{\Cnu}{\C_{\nu}}
\newcommand{\rank}{\operatorname{rank}}
\newcommand{\Top}{\mathcal T}
\newcommand{\infpp}{\sigma}
\newcommand{\Lap}[1]{\widehat{#1}}
\newcommand{\nuassu}{\Bar{\nu}}
\DeclareMathOperator{\Ree}{Re}
\renewcommand{\Re}{\Ree}
\DeclareMathOperator{\Imm}{Im}
\renewcommand{\Im}{\Imm}
\newcommand{\Lprod}{\mathscr L_\nu}
\title{Stabilization of Integral Difference Equations by Solving a Corona Problem}
\author{Adam Braun$^{1}$, Jean Auriol$^{1}$, Lucas Brivadis$^{1}$
\thanks{$^{1}$Universit\'e Paris-Saclay, CNRS, CentraleSup\'elec, Laboratoire des signaux et syst\`emes, 91190, Gif-sur-Yvette, France. {\tt\small firstname.name@l2s.centralesupelec.fr}. This project has received funding from the Agence Nationale de la Recherche (ANR) via grant PANOPLY ANR-23-CE48-0001-01.}%
}
\date{}
\begin{document}
\maketitle
\begin{abstract}
This paper proposes a stabilizing state-feedback control law for vector-valued state systems with scalar control input, governed by a general class of integral difference equations that incorporate both pointwise and distributed input delays.
The proposed controller is expressed through integral operators acting on the state and input histories over a finite time horizon. Closed-loop stability is established by characterizing the controller kernels as solutions of a convolution equation arising from a Corona problem. The existence of such solutions is ensured under a suitable spectral stabilizability condition, and a least-squares procedure is implemented to find them numerically.
The approach extends existing IDE stabilization results to more general settings, allowing for an arbitrary number of pointwise delays affecting both the state and input, without requiring commensurability assumptions. 
\end{abstract}
\section{Introduction}
Integral Difference Equations (IDEs), which incorporate both pointwise and distributed delays, form a broad class of functional equations arising in the modeling of systems with transport phenomena or measurement delays~\cite{Niculescu2001DelayEffects}.
Beyond their intrinsic interest, IDEs have recently gained attention due to their strong connections with partial differential equations (PDEs). In particular, IDEs and interconnected one-dimensional linear hyperbolic balance laws have equivalent stability properties~\cite{auriol_hdr,bastin_coron}, enabling the design of control strategies for PDEs within the IDE framework~\cite{braun2025stabilizationchainhyperbolicpdes,braun_stab_3_2inputs} or~\cite{Jean_Adam_Ouidir_Mathieu}.
Stabilization results for IDEs, under a controllability assumption, have progressively expanded in scope, from scalar systems with a single pointwise delay to more general delay configurations and input structures~\cite{auriol2024stabilizationintegraldelayequations,AuriolSIAM,BPAuriol_fredholm}. A concise comparison of these works is provided in Table~\ref{tab:comparison}.
In this paper, we extend the existing literature by relaxing the controllability requirement to a stabilizability condition, generalizing the framework to vector-valued state systems, and allowing arbitrary numbers of pointwise delays affecting both the state and the input without commensurability assumptions.
Our approach relies on reformulating the control design problem as a Corona problem. We show that exponential stabilizability by means of an autoregressive control scheme can be achieved by solving a convolution equation associated with a Corona-type condition~\cite{Duren2000Hp}, in line with known connections between Corona problems and the controllability of IDEs~\cite{fueyo2025lqapproximatecontrollabilityfrequency}. The resulting control law is explicit and easy to implement, as illustrated by numerical simulations.
\begin{table}[h!]
\centering
\caption{Comparison with prior works}
\label{tab:comparison}
\footnotesize
\setlength{\tabcolsep}{3pt}
\begin{tabular}{llll}
\toprule
Ref. & Pointwise delays & Assumption & Scope \\
\midrule
\cite{auriol2024stabilizationintegraldelayequations} 
& 1 state delay + 1 input delay
& Ctrl. 
& Scalar \\
\cite{AuriolSIAM} 
&  1 state delay + 0 input delay
& Ctrl. 
& Scalar \\
\cite{BPAuriol_fredholm} 
& Multi-delays on state and input
& Ctrl.+comm. 
& Scalar \\
\textbf{This paper} 
& Multi-delays on state and input
& \textbf{Stab.} 
& \textbf{Vector} \\
\bottomrule
\end{tabular}

\vspace{2pt}
\raggedright
\scriptsize
\textit{Ctrl.: controllability; comm.: commensurability; Stab.: stabilizability.}
\end{table}


\section{Mathematical setting and notation}
The set of positive integers is denoted by $\N^*$.
Let $n,m\in\N_{>0}$ and $r\in [1,\infty)$.
The identity matrix in $\R^{n\times n}$ is denoted by $I_n$ (or $I$ when no confusion arises). For a set $C\subset\R$, $\mathbbb{1}_C$ denotes its indicator function. The Lebesgue measure is denoted by $\dt$. The operator norm is denoted by $\|\cdot\|_{op}$, and $|\cdot|$ stands for a matrix norm.
For $C\subset\R$ with nonempty interior, we denote
$
L^r(C,\R^{n\times m}) := \Big\{ M: C\to \R^{n\times m} \ \big|\ \|M\|_{L^r}^r := \int_C |M(t)|^r \dt < \infty \Big\}.$
For $\K\in\{\R,\C\}$ and $\nu>0$, the weighted space $L^2_\nu((0,\infty),\K)$ is the space of measurable functions with values in $\K$ such that
$
\|f\|_{L^2_\nu} := \left(\int_0^\infty e^{2\nu t}|f(t)|^2 \dt\right)^{1/2}<\infty.$
For $x\in L^r(C,\R^n)$ and $t\in C$, the shifted trajectory $x_t$ is defined by $x_t(\tau):=x(t+\tau)$ whenever $t+\tau\in C$. For $a\in\R$, $\delta_a$ denotes the Dirac measure at $a$.
The Laplace transform is denoted by $\mathcal{L}$, and $\widehat{f}$ denotes the Laplace transform of $f$ when defined. For a matrix-valued finite Borel measure $q$, its Laplace transform is
$
\widehat{q}(z) = \int_0^\infty e^{-zs} q(ds),
$
which is well defined on $\C$ if $q$ has compact support. The convolution of $q$ with a function $f$ is defined by
$
(q*f)(t) := \int_\R q(ds)\, f(t-s),
$
whenever it is well defined.
We denote by $\mathbb{D} := \{s\in\C:|s|<1\}$ the unit disk, $\mathbb{T} := \{s\in\C:|s|=1\}$ the unit circle, and for $\nu\ge0$,
$
\C_\nu := \{s\in\C:\Re s>-\nu\}.
$
We denote by $H^2(\mathbb{D})$, $H^\infty(\mathbb{D})$, $H^2(\C_\nu)$, and $H^\infty(\C_\nu)$ the corresponding Hardy spaces (see e.g.~\cite{Partington_2004}). In particular,
$
\|F\|_{H^2(\C_\nu)}^2 := \sup_{x>-\nu} \frac{1}{2\pi}\int_\R |F(x+i\omega)|^2 d\omega, 
\|F\|_{H^\infty(\C_\nu)} := \sup_{s\in\C_\nu} |F(s)|.
$
\section{Problem statement}
\subsection{Integral difference equation}
Let us consider the following \emph{vector-valued state, scalar-input IDE}, defined for any 

$x_0$ in $L^r([-\tau_*, 0], \R^n)$ by
\begin{equation}\label{main IDE}
\begin{cases}
 x(t) - \sum_{k=1}^{\infty} A_k x(t-\tau_k) - 
 \int_0^{\tau_*} N(\eta) x(t-\eta)\,d\eta \\
 =\sum_{j=0}^{\infty} B_j U(t-\theta_j) + 
 \int_0^{\theta_*} M(\eta) U(t-\eta)\,d\eta,
 ~ \textit{$t>0$},\\[4pt]
 x(t) = x_0(t), \quad \textit{$-\tau_* \le t \le 0$}.
\end{cases}
\end{equation}
with $\tau_* >0$ and $\theta_*>0$. 
The sequence of $n\times n$ matrices $(A_k)_{k\in \mathbb{N}^*}$ and $n\times 1$ matrices $(B_j)_{j\in \mathbb{N}}$ are such that  $\sum_{k=1}^\infty  |A_k| +\sum_{j=0}^\infty  |B_j| <\infty$.
The sequences $(\tau_k)_{k\in \mathbb{N}^*}$ and $(\theta_j)_{j\in \mathbb{N}}$ are strictly increasing and for all $k\in \mathbb{N}^*$, $\tau_k \in (0, \tau_*]$, for all $j\in \mathbb{N}$, $\theta_j \in [0, \theta_*]$.
The kernels \(N\) and \(M\) belong respectively to \(L^2((0,\tau_*), \mathbb{R}^{n\times n})\) and \(L^2((0,\theta_*), \mathbb{R}^{n})\).
The function $U$ is the control input, it takes values in $\R$. The input acts on system~\eqref{main IDE} via both pointwise and distributed delays. Its initial condition $U_0$ is in $L^r([-\theta_*, 0], \R)$.
A function $x$ is a solution of~\eqref{main IDE} if, for all $t \ge 0$, we have $x_t \in L^r([-\tau_*,0])$, and system~\eqref{main IDE} holds for almost every $t \geq -\tau_*$.
The well-posedness of system~\eqref{main IDE} in open-loop is ensured by~\cite[Proposition 2]{braun2026spectralexponentialstabilitycriterion}.
\subsection{Objective and Candidate Control Law}

The objective of this article is to exponentially stabilize system \eqref{main IDE} by means of a dynamical autoregressive control law satisfying, for $U_0 \in L^r([-\theta_*,0],\R)$,
\begin{equation}\label{eq:U_form}
\begin{cases}
 U(t) = \sum_{i=1}^n\int_0^{\min(t,S_i)}g_i(\eta)x_i(t-\eta)d\eta \\
 \qquad\quad+\int_0^{\min(t,S_{n+1})} f(\eta)U(t-\eta)d\eta,
 \quad \textit{$t>0$},\\[4pt]
 U(t) = U_0(t), \quad \textit{$-\theta_* \le t \le 0$}.
\end{cases}
\end{equation}
where for all $i\in \{1,\ldots,n\}$, the controller gains $g_i$ and $f$ are square integrable functions compactly supported in $[0,S_i]$, respectively $[0,S_{n+1}]$ to be specified later. The support lengths $S_j$ will also be defined later.
Equation~\eqref{eq:U_form} together with the IDE~\eqref{main IDE} form the closed-loop system, its well-posedness is ensured by~\cite[Proposition 2]{braun2026spectralexponentialstabilitycriterion}.
Our goal is to find the gains such that the closed-loop~\eqref{main IDE}-\eqref{eq:U_form} is exponentially stable,
i.e. that there exist $\nu>0$ and $C\geq 1$ such that for all $(x_0, U_0)\in L^r([-\tau_*, 0], \R^n)\times  L^r([-\theta_*, 0], \R)$, for all $t\geq0$,
$$\|x_t\|_{L^r} +\|U_t\|_{L^r}\leq Ce^{-\nu t}(\|x_0\|_{L^r}+\|U_0\|_{L^r}),$$
where $(x_t, U_t)$ is the corresponding solution to the closed-loop system \eqref{main IDE}-\eqref{eq:U_form}.

\section{Main result}

\subsection{Design assumptions}
\label{section:design assumptions}
The characteristic equation is derived by formally applying the Laplace transform to~\eqref{main IDE} under zero initial conditions; this is sufficient for stability analysis, since the characteristic equation of such IDEs is independent of the initial condition~\cite{braun2026spectralexponentialstabilitycriterion,halebook}. 
This yields
$\Lap q(s)\Lap x(s) = \Lap p(s)\Lap U(s),$
where for all $s\in \C$,
$$
\Lap q(s):=\Delta_0 (s) -\Lap N(s) \in \C^{n\times n},$$
with 
\begin{equation} \label{def:pp}\Delta_0(s):=I-\sum_{k=1}^\infty A_k e^{-\tau_k s}.\end{equation}
and
$$\Lap p(s):=\sum_{j=0}^\infty B_j e^{-\theta_j s} + \Lap M(s)\in \C^{n\times 1}.$$
Notice that $\Lap q$ and $\Lap p$ are the Laplace transforms of the finite compactly supported Borel measures defined by,
\begin{equation}\label{def:mesures_Q_P}
q(\dt) = I_n\delta_0 -\sum^\infty_{k=1} A_k\delta_{\tau_k} - N\dt, p(\dt) = \sum^\infty_{j=0} B_j\delta_{\theta_j} + M\dt.
\end{equation}
As in~\cite{auriol2024stabilizationintegraldelayequations}, we will assume the principal part (i.e without the integral term and with $U\equiv 0$) of the IDE~\eqref{main IDE}
to be exponentially stable. This assumption is \emph{necessary} to avoid
having an infinite number of unstable poles; see~\cite{halebook} and~\cite[Proof of Theorem 8]{AURIOL2019144}. Such poles would make the design of a delay-robust control law impossible (see~\cite[Theorem 1.1]{Logemann1996}).
 By Lemma~\ref{stable_lemma} and~\cite[Proposition A.5]{felipe}, the exponential stability of the principal part of the IDE~\eqref{main IDE} is equivalent to the following assumption.
\begin{assum}[Stability of the principal part]
    \label{assum:pp_stable}
There exist $\nuassu>0$ and $\infpp>0$ such that, for all $s\in \C_{\nuassu}$, $$
       | \det \Delta_0(s)|>\infpp.$$
\end{assum}
Therefore, the main difficulty in stabilizing the integral delay equation~\eqref{main IDE} stems from the presence of the term
$
\int_{0}^{\tau_*} N(\eta)x(t-\eta)\,d\eta,$
which may have a destabilizing effect on the system dynamics. The proposed strategy consists in partially eliminating this integral contribution through an appropriate choice of the control law~$U$.
To do so, we will use the following assumption that extends the controllability condition of~\cite[Assumption 1]{auriol2024stabilizationintegraldelayequations} to the vector-valued, multi-delay setting. However, it is required to hold only on the right half-plane $\C_{\nuassu}$ for $\nuassu > 0$, thereby constituting a strong stabilizability condition, consistent with the case without distributed delays (see~\cite[Theorem 3.1]{hale_feedback_2} and the discussion after equation~(3.3) in the same reference).
\begin{assum}[Spectral stabilizability]\label{assum:spectral_stabilisabilité}
   \modif{For $\nuassu$ as in Assumption~\ref{assum:pp_stable}, for all $s\in \C_{\nuassu}$},
    $$\rank [\Lap q(s),-\Lap p(s)] = n.$$
\end{assum}
\subsection{Result statement and discussion}

\begin{theorem}\label{main_thm} Under Assumptions \ref{assum:pp_stable} and \ref{assum:spectral_stabilisabilité},
there exist controller gains $$f, g_1, \dots, g_n$$ in $L^2((0,\infty),\R)$, compactly supported, such that the closed-loop system \eqref{main IDE}-\eqref{eq:U_form} is exponentially stable.
\end{theorem}
In previous contributions (see Table~\ref{tab:comparison}), the gains were obtained via the inversion of a Fredholm operator. The difficulty with this approach was to show that the considered operator was Fredholm. In contrast, in this paper we formulate the gain computation as a Corona problem, yielding a simpler and more natural approach applicable to a broader class of IDEs.
\begin{remark}
 Theorem~\ref{main_thm} remains valid when the initial condition of system~\eqref{main IDE} is taken to be a bounded Borel-measurable function or a function of bounded variation.
Moreover, the matrix-valued kernels \(N\) and \(M\) could also be taken in \(L^1\), by approximating them with $L^2$ kernels.
\end{remark}
The rest of the paper
is dedicated to the proof of Theorem~\ref{main_thm}.
\section{Proof of the main result}
\subsection{The characteristic equation}
Let $Y(t):= (x(t), U(t)) \in \R^{n+1}$
be the extended variable associated to the closed-loop system~\eqref{main IDE}-\eqref{eq:U_form}. Formally taking the Laplace transform of~\eqref{main IDE}-\eqref{eq:U_form}, we obtain, for all $s\in \C$,
$A(s)\widehat Y(s)=0,$
where,
\begin{equation}
    \label{matrix_characteristic}
    A(s):= \begin{bmatrix}\Lap q(s)& -\Lap p(s)\\
    -\ghat(s) &1 - \fhat(s)
    \end{bmatrix},
\end{equation}
with 
$$\ghat(s) = \begin{bmatrix}
    \ghat_1(s)~\ldots~\ghat_n(s)
\end{bmatrix} \in \C^{1\times n}.$$
    The characteristic equation associated to the closed-loop system~\eqref{main IDE}-\eqref{eq:U_form} is $\det A = 0$.
\begin{lemma}{\cite[Theorem 6]{braun2026spectralexponentialstabilitycriterion}} \label{stable_lemma}
    The closed-loop system~\eqref{main IDE}-\eqref{eq:U_form} is exponentially stable if and only if there exists $\nu>0$ such that
    $\det A(s)\neq 0$ for all $s\in \Cnu$.
\end{lemma}
For all $j\in \{1,\ldots,n+1\}$ let $r_j$ be the minor of order $n$, i.e. the determinant of the $n\times n$ submatrix of $[\Lap q, -\Lap p]$ obtained by removing the column $j$. Then Assumption~\ref{assum:spectral_stabilisabilité} implies that for all $s\in \C_{\nuassu}$,
    \begin{equation} \label{eq:rj_rang_1}\rank[r_1(s)~\ldots~r_{n+1}(s)] =1.\end{equation}
The following lemma is a consequence of Cauchy-Schwarz inequality.
\begin{lemma}\label{lem:q_p_Hinfty_infinite} For all \(\nu>0\),
each entry of \(\Lap q\) and each component of \(\Lap p\) belongs to \(H^\infty(\Cnu)\).
\end{lemma}

A direct consequence of Lemma~\ref{lem:q_p_Hinfty_infinite} is that for all \(j \in \{1,\dots,n+1\}\), 
$r_j \in H^\infty(\C_{\nuassu}).$
Furthermore, each \(r_j\) can be expressed as a finite sum of products of Laplace transforms of finite, compactly supported, real-valued Borel measures. In the time domain, this corresponds to a finite sum of convolutions of such measures. Since this class is closed under convolution and finite sums, \(r_j\) is itself the Laplace transform of a finite, compactly supported, real-valued Borel measure \(R_j\), i.e., $
r_j(s) = \int_{[0,\infty)} e^{-st}\, dR_j(t)$ for all $s \in \C_{\nuassu}$.
We decompose the characteristic equation of~\eqref{main IDE}-\eqref{eq:U_form} into stable and unstable parts.
\begin{prop} \label{prop:calcul_det_A} For $\nuassu $ as in Assumption~\ref{assum:spectral_stabilisabilité}, there exists $T>0$ and $N_T$ in \(L^2((0,T),\R)\) such that
$\Ntarget\in H^2(\C_{\nuassu})$ and for all $s\in \C_{\nuassu}$
    \begin{align}\label{eq:characteristic}\det A(s) &= \det \Delta_0(s) \\ \nonumber&+\Ntarget(s)-\fhat(s)r_{n+1}(s)-\sum_{j=1}^n (-1)^{n+1+j} r_j(s)\ghat_j(s),\end{align}
    where $(r_j)_j$ are the minors of $[\Lap q, -\Lap p]$. 
\end{prop}
\begin{proof} Let $s\in \C_{\nuassu}$.
    By developing the determinant of $A(s)$ along the last row, we obtain
    $$\det A(s) = (1-\fhat(s))r_{n+1}(s)-\sum_{j=1}^n (-1)^{n+1+j} r_j(s)\ghat_j(s).$$ 
     We now compute $r_{n+1}(s) = \det \Lap q(s)$ using the multilinearity of the determinant acting on the columns of $$\Lap q =[\Delta_{0,1} -\Nhat_1~\ldots~\Delta_{0,n} -\Nhat_n],$$
     where $\Nhat_i$ is the Laplace transform of the $i$-th column of $N\in \R^{n\times n}$ and $\Delta_{0,i}$ is the $i$-th column of $\Delta_0$ defined in~\eqref{def:pp}. We obtain
    $\det \Lap q(s) = \det \Delta_0 + \Ntarget(s),$
    with
\begin{equation} \label{def:Ntarget}
\Lap N_T(s) := (-1)^n\det \Nhat(s) + \sum_{\substack{L \subset \{1,\dots,n\} \\ L \neq \varnothing,\, L \neq \{1,\dots,n\}}} \det M_L(s),
\end{equation}
    where $M_L(s)$ is an $n\times n$ complex matrix with columns $m_j$ defined for all $j\in\{1,\ldots,n\} $ by,
    \begin{equation*} m_j = \begin{cases}
        \Delta_{0,j}\qquad \textit{if $j\notin L$} \\
        -\Nhat_j(s)~~\textit{if $j\in L$}.
    \end{cases}\end{equation*}
Then, \(\Lap N_T \in H^2(\C_{\nuassu})\) and is the Laplace transform of a compactly supported function $N_T \in L^2((0,T),\R)$.

Indeed, each entry of \(\Lap N\) is the Laplace transform of a compactly supported function in \(L^2((0,\infty), \R)\). Hence, its determinant corresponds in the time domain to finite sums of convolutions of compactly supported \(L^2\)-functions. Since convolution preserves both compact support and \(L^2\)-regularity, it follows that \(\det(\Nhat) \in H^2(\C_{\nuassu})\). Similarly, each \(\det M_L(s)\) is a finite sum of products of Laplace transforms of convolutions of measures (arising from the inverse Laplace transform of \(\Delta_0\)~\eqref{def:pp}) and entries of \(\Nhat\). Therefore, it is itself the Laplace transform of a causal, compactly supported function \(N_T\in L^2_{\nuassu}((0,\infty), \R)\). 
 By Proposition~\ref{prop:paley}, \(\Ntarget \in H^2(\C_{\nuassu})\).

\end{proof}
\begin{remark} \label{rmk:n=1}
    If the state dimension is $n=1$, then $N_T=-N$, $R_1(\dt) = -p(\dt)$ and $R_2(\dt) = q(\dt)$ (defined in~\eqref{def:mesures_Q_P}).
\end{remark}
\subsection{A Corona problem for the gains}
By Lemma~\ref{stable_lemma} and in view of Assumption~\ref{assum:pp_stable} and~Proposition~\ref{prop:calcul_det_A}, the exponential stabilization of the closed-loop system~\eqref{main IDE}-\eqref{eq:U_form} reduces to the construction of gains such that there exist \ab{$\mu, \nu$ such that $0<\mu< \nu<\nuassu$} (with $\nuassu$ given in Assumption~\ref{assum:pp_stable}) for which, for all \ab{$s\in\C_{\nu -\mu}$,}
\begin{equation}
    \label{eq:notre_corona_pb}
   |\fhat(s)r_{n+1}(s)+\sum_{j=1}^n (-1)^{n+1+j} r_j(s)\ghat_j(s)-\Ntarget(s)|\leq \infpp.
\end{equation}
This is an approximate Corona problem~\cite[Chapter 12]{Duren2000Hp} (in the sense that the usual Corona problem would be \eqref{eq:notre_corona_pb} with $\infpp=0$).
In Appendix~\ref{appendix:hardy}, we recall some basic facts about Hardy spaces and in Appendix~\ref{appendix:Corona} we state a Corona Theorem. Set
$\Lprod:= \prod_{j= 1}^{n+1} L^2_\nu((0, \infty), \C),$
it is a Hilbert space equipped with the sum of the natural scalar product on $L^2_\nu((0, \infty),\C).$ By Propositions~\ref{prop:paley} \ab{and~\ref{prop:h2_pointwise_bound}}, Problem~\eqref{eq:notre_corona_pb} can be reformulated in the time domain. We want to find $(g,f) \in \Lprod$ real-valued and compactly supported such that,
\ab{\begin{equation} \label{eq:corona_final}\|T_\nu(g, f) -N_T\|_{L^2_\nu} \leq \infpp\sqrt{2\mu} ,\end{equation}}
where
\begin{equation}\label{def_Tnu}
T_\nu:\Lprod\to L^2_\nu((0,\infty),\C),~
T_\nu(v)=\sum_{j=1}^{n+1}(-1)^{n+1+j} R_j*v_j,
\end{equation}
with $R_j = \mathcal{L}^{-1}(r_j).$
Set  $$\|R^\nu\|_{TV} := \sqrt{\|R^\nu_1\|_{TV}^2+\ldots +\|R_{n+1}^\nu\|_{TV}^2},$$
with $R^\nu_j(\dt) := e^{\nu t}R_j(\dt)$ still a finite Borel measure because $R_j$ is compactly supported. Using~\cite[Theorem 3.5, Section 4.3]{GripenbergLondenStaffans1990} (a Young's convolution inequality in weighted spaces), we obtain for all $f\in L^2_\nu((0,\infty), \C)$
$$
\|R_j* f\|_{L^2_\nu}\leq\|R_j^\nu\|_{\mathrm{TV}}\|f\|_{L^2_\nu}.$$
Hence, by Cauchy-Schwarz in $\R^{n+1}$, for all $v\in\Lprod$,
$$\|T_\nu(v)\|_{L^2_\nu} \leq \sum_{j=1}^{n+1}\|R_j^\nu\|_{TV}\|v_j\|_{L^2_\nu}\leq \|R^\nu\|_{TV} \|v\|_{\Lprod}.$$
Thus $T_\nu$ is bounded.
We construct gains satisfying~\eqref{eq:notre_corona_pb} via Theorem~\ref{thm:toeplitz_corona} \modif{(see Appendix~\ref{appendix:Corona})}. We now establish the conditions for its application.
\begin{lemma}
    \label{lem:hyp_corona}
    Under Assumptions~\ref{assum:pp_stable} and
    \ref{assum:spectral_stabilisabilité}, \ab{for all
    \(
        0<\nu_*<\nuassu,
    \)}
    where $\nuassu$ is given by Assumption~\ref{assum:pp_stable}, one has
    \(
        d
        :=
        \inf_{s\in\C_{\nu_*}}
        \sum_{j=1}^{n+1}|r_j(s)|
        >0.
    \)
\end{lemma}

\begin{proof}
Assume by contradiction that there exist
$0<\nu_*<\nuassu$ and a sequence
$(s_k)_{k\in\N}\subset \C_{\nu_*}$ such that
$ |r_j(s_k)| \to 0,~\forall j\in\{1,\dots,n+1\}.$
\modif{We first show that $|s_k|\to\infty$. Indeed, if $(s_k)$ were bounded,
then, up to extracting a subsequence, $s_k\to s_\infty$ with
\(
    \Re s_\infty \geq -\nu_*.
\)
Since $\nu_*<\nuassu$, we have
$s_\infty\in\C_{\nuassu}$. By continuity of the functions $r_j$, we would
obtain a common zero
which contradicts~\eqref{eq:rj_rang_1}.} Hence $|s_k|\to\infty$.
From the proof of Proposition~\ref{prop:calcul_det_A}, we have $$ r_{n+1}(s)=\det \Delta_0(s)+\Ntarget(s). $$ Up to extraction, either \(|\Im s_k|\to\infty\) or \(\Re s_k\to\infty\), and thus $ \Ntarget(s_k)\to 0$ by the Riemann-Lebesgue lemma or dominated convergence. Hence, $ \det \Delta_0(s_k)\to 0,$ which contradicts Assumption~\ref{assum:pp_stable}.
\end{proof}

The following theorem is the time domain analog of Theorem~\ref{thm:toeplitz_corona} \modif{(see Appendix~\ref{appendix:Corona})}.
\begin{theorem} \label{thm:temporel_corona}
Under Assumptions~\ref{assum:pp_stable} and~\ref{assum:spectral_stabilisabilité}, let ${\nu_*}$ and $d$ as in Lemma~\ref{lem:hyp_corona}. Then, the operator $T_{{\nu_*}}$ (defined in~\eqref{def_Tnu}) is surjective.
In particular, for all functions $N_*\in L^2_{{\nu_*}}((0,\infty),\R)$, there exists a unique $v ^* \in \mathscr L_{{\nu_*}}$ real valued, such that
\begin{equation}
    \label{inf_temporal}
    \|v^*\|_{\mathscr L_{{\nu_*}}} = \inf_{v \in T_{\nu_*}^{-1}(\{N_*\})}  \|v\|_{\mathscr L_{{\nu_*}}},
\end{equation}
and there exists $C(n,d)>0$ such that
\begin{equation}
    \label{eq:bound_norm_minimiseur_temporal}
    \|v^*\|_{\mathscr L_{{\nu_*}}}\leq C(n,d) \| N_*\|_{L^2_{{\nu_*}}}.
\end{equation}
Moreover, for all $0<\nu <{\nu_*}$ and all $\varepsilon>0$, setting for all $j\in \{1,\ldots,n\}$ $g_j := v_j^*\mathbbb{1}_{[0,S_j]}, f:=v_{n+1}^*\mathbbb{1}_{[0,S_{n+1}]}$ with,
\small
\begin{equation}
    \label{eq:supports_Si}
    \min_{j \in \{1,\ldots,n+1\}} S_j\geq \frac{1}{{\nu_*} -\nu}\ln\Big(\frac{1}{\varepsilon}\|R^\nu\|_{TV}C(n,d)\| N_*\|_{L_{{\nu_*}}^2}\modif{+1}\Big),
\end{equation}
\normalsize
we have,
$\|T_\nu(g, f) -N_*\|_{L^2_{\nu}} \leq \varepsilon.$

\end{theorem}

\begin{proof}
We reformulate the problem in $H^2(\D)$ using the tools introduced in Appendix~\ref{appendix:hardy}.
For all $j\in \{1,\ldots,n+1\}$, set 
$
\widetilde r_j(s):=(-1)^{n+1+j}r_j(\phi_{\nu_*}^{-1}(s)).$
Then $\widetilde r_j\in H^\infty(\mathbb D)$ because $r_j\in H^\infty(\mathbb C_{\nu_*})$.
Moreover, 
\begin{equation} \label{eq: egalite_iso}\inf_{s\in \D} \sum_{j=1}^{n+1} |\widetilde r_j(s)| = \inf_{s\in \C_{\nu_*}} \sum_{j=1}^{n+1} |r_j(s)|=d>0~ \textit{by Lemma~\ref{lem:hyp_corona}.}\end{equation}
We have
$$
U_{{\nu_*}}\big(\sum_{j=1}^{n+1}(-1)^{n+1+j} r_jH^2(\mathbb C_{{\nu_*}})\big)
=
 \sum_{j=1}^{n+1} \widetilde r_j\,H^2(\mathbb D),$$
 and, 
$$
\mathcal L\big(\operatorname{Ran}(T_{{\nu_*}})\big)=\sum_{j=1}^{n+1} (-1)^{n+1+j}r_j(s)H^2(\mathbb C_{{\nu_*}}).
$$
Hence by Proposition~\ref{prop:paley} and because $U_{{\nu_*}}$ is an isometric isomorphism,
\begin{align}
&\operatorname{Ran}(T_{{\nu_*}})=L^2_{{\nu_*}}((0,\infty),\C)\nonumber
\\&\iff \label{assertion_vraie}
\sum_{j=1}^{n+1} \widetilde r_j(s)H^2(\D)=H^2(\D).
\end{align}
Using~\eqref{eq: egalite_iso}, equation~\eqref{assertion_vraie} is true as a consequence of Theorem~\ref{thm:toeplitz_corona} \modif{(see Appendix~\ref{appendix:Corona})}. 
\modif{Let $N_*\in L^2_{{\nu_*}}((0,\infty),\R)$ and set
\(
    \widetilde N_* := U_{{\nu_*}}(\mathcal L N_*) \in H^2(\D).
\)
By Theorem~\ref{thm:toeplitz_corona}, there exists a unique minimizer
$h^*\in\mathcal H$ satisfying \eqref{eq:minimiseur} and
\eqref{eq:bound_norm_minimiseur}. Define
\(
    v^* := \mathcal L^{-1}U_{{\nu_*}}^{-1}(h^*)
\)
component-wise. Since $U_{{\nu_*}}$ and
$\mathcal L:L^2_{{\nu_*}}((0,\infty),\C)\to H^2(\C_{{\nu_*}})$ are isometric
isomorphisms, we obtain \eqref{inf_temporal} and
\eqref{eq:bound_norm_minimiseur_temporal}.}
\modif{Furthermore, $v^*$ is real-valued. Indeed, since $T_{{\nu_*}}$ has real
coefficients and $N_*$ is real-valued, we have
\(
    T_{{\nu_*}}(\Re v^*)=\Re(T_{{\nu_*}} v^*)=N_*.
\)
Moreover,
\(
    \|\Re v^*\|_{\mathscr L_{{\nu_*}}}\leq \|v^*\|_{\mathscr L_{{\nu_*}}}.
\)
By the uniqueness of the minimizer in \eqref{inf_temporal}, it follows that
\(
    v^*=\Re v^*.
\)
}
Now, \ab{let $0<\nu< {\nu_*}$} and $\varepsilon>0$, let $(S_j)\in \R^{n+1}$ be such that~\eqref{eq:supports_Si} holds.
Notice that, $\forall v\in L_{\nu_*}^2((0,\infty),\C) ~\forall S\geq 0$,
\ab{\begin{equation}\label{eq:weighted_tail}
\|v\|_{L^2_\nu(S,\infty)}\le e^{-({\nu_*}-\nu) S}\|v\|_{L^2_{{\nu_*}}(0,\infty)}.
\end{equation}}
Then, using the boundedness of $T_\nu$,
\begin{align*}
   & \|T_\nu(g,f)-N_*\|_{L^2_\nu} = \|T_\nu(g,f)-T_\nu(v^*)\|_{L^2_\nu}\\
    &\leq \|R^\nu\|_{TV} \|(g,f)-v^*\|_{\Lprod}\\
    &= 
\|R^\nu\|_{TV}
\big(
    \sum_{j=1}^{n+1}
    \|v_j^*\|_{L^2_\nu(S_j,\infty)}^2
\big)^{1/2}\\
    &\leq \max_{1\leq j\leq n+1}e^{-({\nu_*}-\nu)S_j} \|R^\nu\|_{TV}\big(
    \sum_{j=1}^{n+1}
    \|v_j^*\|_{L^2_{{\nu_*}}}^2
\big)^{1/2} \textit{by~\eqref{eq:weighted_tail}}\\
    &\leq\max_{1\leq j\leq n+1}e^{-({\nu_*}-\nu)S_j}\|R^\nu\|_{TV}C(n,d)\| N_*\|_{L_{{\nu_*}}^2}\textit{by~\eqref{eq:bound_norm_minimiseur_temporal}}\\
    &\leq\varepsilon \quad \textit{by~\eqref{eq:supports_Si}}.
\end{align*}
\end{proof}
\subsection{Finding the gains}
The (real-valued) gains $(g,f)$ are given by Theorem~\ref{thm:temporel_corona} with \modif{$N_* = N_T$} and \ab{$\varepsilon=\infpp\sqrt{2\mu}$} (see Assumption~\ref{assum:pp_stable}), \ab{where $0<\mu<\nu<\nu_*$. They satisfy~\eqref{eq:corona_final}.}
Hence, we found a solution to the Corona problem~\eqref{eq:notre_corona_pb}.
Therefore, the characteristic equation~\eqref{eq:characteristic} associated to the closed-loop system~\eqref{main IDE}-\eqref{eq:U_form}, has no solutions on $\C_{\nu-\mu}$. The conclusion follows from Lemma~\ref{stable_lemma}.
\section{Numerical simulations}
In this section, we illustrate the proposed methodology by means of numerical simulations. The code required to reproduce these simulations is available in~\cite{braun2026stabilization_code}.
\modif{We consider system~\eqref{main IDE} with
$$\tau_1=1,\tau_2=\tau_\ast=2,$$
\[
A_1=
\begin{pmatrix}
0.20 & 0.05\\
0.10 & 0.15
\end{pmatrix},
A_2=
\begin{pmatrix}
0.10 & 0.02\\
0.03 & 0.08
\end{pmatrix},
\]
\[
B_0=
\begin{pmatrix}
1\\
0.5
\end{pmatrix},
B_1=
\begin{pmatrix}
0.2\\
1.2
\end{pmatrix},~
\theta_0=0.8,~\theta_1=\theta_\ast=\frac{\pi}{2},
\]
\[
N(\eta)=
\begin{pmatrix}
\sin(\eta) & 0.04\cos(\eta)\\
0.06\sin(2\eta) & \sin(\eta)
\end{pmatrix},
M(\eta)=
\begin{pmatrix}
\sqrt{2\eta}\\
0.5\sqrt{\eta}
\end{pmatrix}.
\]
The initial condition is defined for all
$t\in [-2,0],$
by
\[
x_0(t)=
\begin{pmatrix}
0.2\cos(2t)e^{t/3}\\
-0.15\sin(1.5t)e^{t/4}
\end{pmatrix}.
\]}
This \modif{academic example} falls outside the frameworks considered in~\cite{auriol2024stabilizationintegraldelayequations}-\cite{BPAuriol_fredholm}.
We solve the time-domain formulation of~\eqref{eq:notre_corona_pb}
by discretizing the measures
\(
R_j = \mathcal{L}^{-1} (r_j)
\)
and applying a least-squares procedure. The resulting gains are then truncated in order to obtain compact support through an additional minimization step (see Figure~\ref{fig:gains}). This approach proves to be both simple and robust. 
Figure~\ref{fig:dynamic} shows the evolution of the principal part (without the distributed state delay) \modif{in open-loop with \(U \equiv 0\)}, the open-loop system~\eqref{main IDE} with \(U \equiv 0\), and the closed-loop system~\eqref{main IDE}-\eqref{eq:U_form}.
\begin{figure}[h!]
    \centering
    \subfloat[The controller gains\label{fig:gains}]{
        \includegraphics[width=0.9\linewidth]{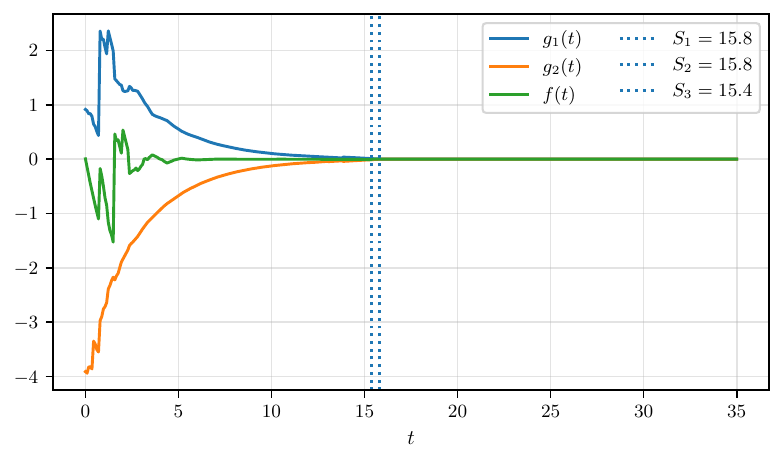}
    }

    \subfloat[Dynamic behavior\label{fig:dynamic}]{
        \includegraphics[width=0.9\linewidth]{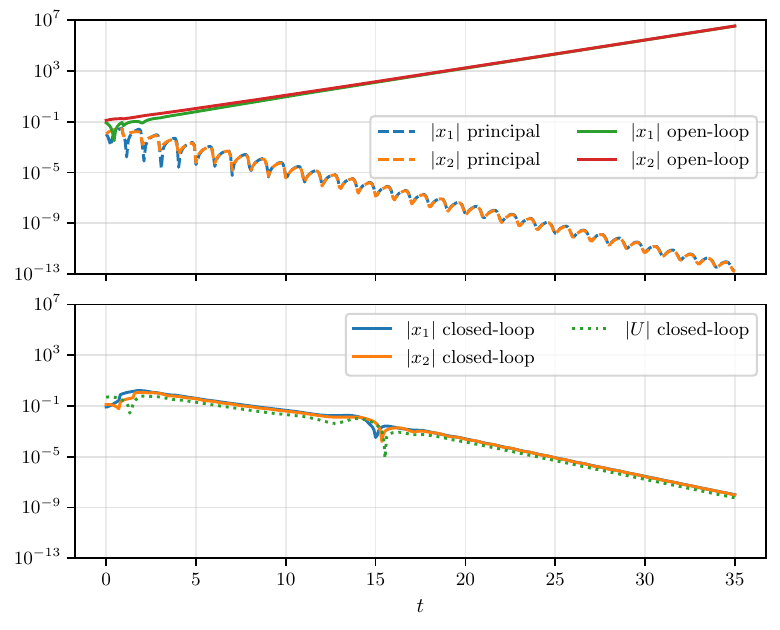}
    }
    \caption{Numerical simulation of system~\eqref{main IDE} with two pointwise delays for $x$ and $U$.}
    \label{fig:all}
\end{figure}
\section{Conclusion}
We have proposed a stabilization method for a general class of vector-valued states, scalar-input IDEs. The control law is expressed as an integral operator acting on the state and input histories, with kernels obtained as solutions to a convolution equation arising from a Corona problem.
Future work will address the vector-valued multi-input setting, which is expected to enable the stabilization of general networks of hyperbolic PDEs~\cite{auriol_hdr}.
\section*{Acknowledgement}
The authors thank Sébastien Fueyo for insightful discussions on the Corona problem and the controllability of IDEs.
\appendix

\section{About the Corona problem}
\subsection{Some reminders about Hardy spaces}
\label{appendix:hardy}
Let $\nu >0$.
We begin by constructing an isometric isomorphism between \(H^2(\Cnu)\) and \(H^2(\mathbb D)\), which will be useful for reformulating the Corona problem~\eqref{eq:notre_corona_pb} on the unit disk \(\mathbb D\).
Let $\phi_\nu:\Cnu\to\mathbb D$ be the Cayley transform for the shifted half-plane:
\begin{equation} \label{eq:Cailey}
\phi_\nu(s)=\frac{(s+\nu)-1}{(s+\nu)+1},\qquad 
\phi_\nu^{-1}(s)=-\nu+\frac{1+s}{1-s}.
\end{equation}
This transform is in particular, a Möbius transform (see e.g.~\cite[chapter III, §3]{Conway1978} and~\cite[page 189, Section 6.3]{GreeneKrantz2006}).
The mapping $\phi_\nu^{-1}$ is a biholomorphism between $\mathbb D$ and $\Cnu$. 
 Let $U_\nu:H^2(\Cnu)\to H^2(\mathbb D)$ be the standard isometric isomorphism associated with $\phi_\nu$ (see e.g.~\cite[Theorem 1.2.5, Chapter 1]{Partington_2004}), defined for $f \in H^2(\Cnu)$ by
$$
(U_\nu f)(s)= \frac{\sqrt{2}}{1-s}\,f(\phi_\nu^{-1}(s)),~\forall s\in \D.$$
It satisfies, for all $\psi\in H^\infty(\Cnu)$, all
$f\in H^2(\Cnu)$, and all $s\in \D$,
$$
U_\nu(\psi f)(s)=(\psi\circ \phi_\nu^{-1})(s)\,(U_\nu f)(s).$$
\begin{prop}[Paley-Wiener (shifted half-plane)] \label{prop:paley} Let $\nu>0$.
The Laplace transform is an isometric isomorphism between $L^2_\nu((0,\infty), \C)$ and $
H^2(\Cnu).$
\end{prop}
\begin{proof}
Let \(f\in L^2_\nu((0,\infty),\C)\) and set \(g(t):=e^{\nu t}f(t)\), then one can apply the Paley-Wiener Theorem~\cite[Definition 1.2.4 and discussion, Chapter 1]{Partington_2004} to \ab{$g \in L^2((0, \infty), \C)$}.
\end{proof}
\ab{The proof of the following result can be adapted from~\cite[Chapter VI]{Duren2000Hp}.
\begin{prop} \label{prop:h2_pointwise_bound}
    For all $f\in H^2(\Cnu)$, all $0<\mu <\nu$, and all $s\in \C_{\nu -\mu},$
\begin{equation}\label{eq:f_point_bound}|f(s)|
\le
\frac{1}{\sqrt{2\mu}}\,\|f\|_{H^2(\mathbb{C}_\nu)}.
\end{equation}
\end{prop}
\begin{proof}
The Hardy space \(H^2(\C_\nu)\) is a reproducing kernel Hilbert space (see~\cite{Duren2000Hp}), with
kernels
$$
k_s(\cdot )=\frac{1}{\cdot+\overline{s}+2\nu}.
$$
Meaning, for all \(f\in H^2(\C_\nu)\), and all $s\in \C_\nu$
\(
f(s)=\langle f,k_s\rangle_{H^2(\C_\nu)}.
\)
By Cauchy-Schwarz inequality and because
$$
\|k_s\|_{H^2(C_\nu)}^2=k_s(s)
=\frac{1}{2(\operatorname{Re}s+\nu)},
$$
we obtain
\[
|f(s)|
\leq
\frac{1}{\sqrt{2(\operatorname{Re}s+\nu)}}\,
\|f\|_{H^2(C_\nu)}.
\]
Consequently if \(s\in C_{\nu-\mu}\), we obtain~\eqref{eq:f_point_bound}.
\end{proof}}
The space $H^2(\D)$ is a closed subspace of the Hilbert space $L^2(\T)$~\cite[Theorem 1.2.2, Chapter 1]{Partington_2004}. 
The scalar product on $H^2(\D)$ is the natural one induced by $L^2(\T)$, namely, for $f$ and $g$ in $ H^2(\D)$, 
$$\langle f, g\rangle_{H^2}:= \langle f, g\rangle_{L^2(\mathbb{T})}
=  \frac{1}{2\pi} \int_0^{2\pi} f(e^{i\theta}) \overline{g(e^{i\theta})}\, d\theta.$$
 In the following, we set
$
\mathcal{H} :=(H^2(\mathbb{D}))^{n+1},
$
which we view as a subspace of \((L^2(\mathbb{T}))^{n+1}\). Since \(H^2(\mathbb{D})\) is a closed subspace of \(L^2(\mathbb{T})\), it follows that \(\mathcal{H}\) is a closed subspace of \(\prod_{j=1}^{n+1} L^2(\mathbb{T})\), and hence a Hilbert space.
The scalar product on \(\mathcal{H}\) is the natural one induced by \(L^2(\mathbb{T})\), namely, for \(f = (f_1,\dots,f_{n+1})\) and \(g = (g_1,\dots,g_{n+1})\) in \(\mathcal{H}\),
$
\langle f, g \rangle_{\mathcal{H}} 
:= \sum_{j=1}^{n+1} \langle f_j, g_j \rangle_{L^2(\mathbb{T})}.$
\subsection{Carleson's Corona theorem}

\label{appendix:Corona}
\ab{In this subsection, we state a particular case of the well-known Carleson Corona theorem on the unit disk $\D$~\cite{Carleson1962}.
Let $\widetilde r_j\in H^\infty(\D) $ for all $j \in \{1,\ldots,n+1\}$.
Set
\begin{equation}
    \label{def:operator_T} 
    \Top:\mathcal{H} \to H^2(\D), \quad \Top h = \sum_{j=1}^{n+1} \widetilde r_j h_j.
\end{equation}
By Cauchy-Schwarz inequality, \(\Top\) is bounded and satisfies
$$
\|\Top h\|_{H^2}
\le \left(\sum_{j=1}^{n+1}\|\widetilde r_j\|_{H^\infty}^2\right)^{1/2}\,\|h\|_{\mathcal H}, \forall h\in\mathcal H.$$
\begin{theorem}[Corona in $H^2$] \label{thm:toeplitz_corona}
Assume $$d:=\inf_{z\in\D} \sum_{j=1}^{n+1} |\widetilde r_j(z)|>0.$$ Then, for all $\widetilde N_* \in  H^2(\D)$, there exists a unique $h^* \in\mathcal{H}$ such that
  \begin{equation}
      \label{eq:minimiseur}
      \|h^*\|_{\mathcal{H}} = \inf_{h\in \Top^{-1}(\{\widetilde N_*\})} \|h\|_{\mathcal{H}}.
  \end{equation}
Moreover, there exists $C(n,d)>0$ independent of $\widetilde N_*$ such that
\begin{equation}
    \label{eq:bound_norm_minimiseur}
    \|h^*\|_{\mathcal{H}}\leq C(n,d)\|\widetilde N_*\|_{H^2(\D)}.
\end{equation}
\end{theorem}
\begin{proof}
Let $\widetilde N_* \in  H^2(\D)$.
Since $\widetilde r_j \in H^\infty(\D)$ for all $j\in \{1,\ldots, n+1\}$, we can apply~\cite[Theorem~2.1 Section~VIII.2]{Garnett2007} to obtain $\phi_j\in  H^\infty(\D)$ for all $\{j\in 1,\ldots, n+1\}$ such that
\[
\sum_{j=1}^{n+1} \widetilde r_j(s) \phi_j(s) =1, \forall s\in \D,
\]
and such that there exists $C(n,d)>0$ (see also~\cite[Theorem 1.1]{andersson2000estimates} for a sharper bound) such that
\begin{equation}
    \label{eq:bound_phi}
    \sum_{j=1}^{n+1}\|\phi_j\|_{H^\infty(\D)} \leq C(n,d).
\end{equation}
Set $h_j:= \phi_j \widetilde N_*$ for all $j\in \{1,\ldots, n+1\}$. Then, $h=(h_1,\ldots,h_{n+1}) \in \mathcal H$,
\[\mathcal T h =\widetilde N_*,\]
and
\begin{equation*}
    \|h\|_{\mathcal{H}}\leq C(n,d)\|\widetilde N_*\|_{H^2(\D)}.
\end{equation*}
Hence $\Top ^{-1}(\{\widetilde N_*\})\neq \varnothing$, furthermore, by linearity and continuity of $\Top $~\eqref{def:operator_T}, the set $\Top ^{-1}(\{\widetilde N\})\neq \varnothing$ is convex and closed.
By Hilbert's projection theorem, there exists a unique minimizer $h^*$ in the sense of~\eqref{eq:minimiseur}.
Because $\|h^*\|_\mathcal H \leq \|h\|_\mathcal H $, we have~\eqref{eq:bound_norm_minimiseur}.
\end{proof}}
\modif{
\begin{remark}
    Under the slightly stronger assumption that there exists $d(n)>0$ such that
$\mathcal{T}\mathcal{T}^* \geq d(n)\, I$,
where $I : H^2(\mathbb{D}) \to H^2(\mathbb{D})$ is the identity operator, 
the bound in~\eqref{eq:bound_norm_minimiseur} can be sharpened. Specifically, 
by~\cite[Theorem 8.53, Section 8.4]{agler2023pick}, one obtains
\[
    C(n,d) = \frac{1}{\sqrt{d(n)}}.
\]
\end{remark}}
\bibliographystyle{abbrv}
\bibliography{references}
\end{document}